\documentclass[reqno]{amsart}

\usepackage{cite}
\usepackage{a4wide,calrsfs}
\usepackage{amsmath,bbm}
\usepackage{color}
\usepackage{epsfig}
\usepackage{subfigure}
\usepackage{amssymb}
\usepackage{enumerate}
\usepackage{graphicx}
\usepackage{subfigure}
\usepackage{amsmath}
\usepackage{rotating}
\usepackage{stmaryrd}
\usepackage{upgreek}  
\usepackage{amsthm}
\usepackage{amsopn}
\usepackage{upgreek}

\newcommand{\R}{\mathbb R}

\newcommand{\Z}{\mathbb{Z}}

\newcommand{\C}{\mathbb C}

\begin{document}
\title[Dichotomies and  asymptotic equivalence]
{Dichotomies and  asymptotic equivalence in alternately advanced and
delayed differential systems}

\author[A. Coronel  \& C. Maul\'en \& M. Pinto \& D. Sep\'ulveda]
{An{\'\i}bal Coronel \and Christopher Maul\'en\and  Manuel Pinto \and  Daniel Sep\'ulveda}  

\address{An\'ibal Coronel \newline
GMA, Departamento de Ciencias B\'asicas, Facultad de Ciencias, 
Universidad del B\'{\i}o-B\'{\i}o, Campus Fernando May, Chill\'{a}n, Chile.}
\email{acoronel@ubiobio.cl}

\address{Christopher Maul\'en \newline
Departamento de Matem\'aticas, Facultad de Ciencias, Universidad de Chile}
\email{christoph.maulen.math@gmail.cl}

\address{Manuel Pinto \newline
Departamento de Matem\'aticas, Facultad de Ciencias, Universidad de Chile}
\email{pintoj.uchile@gmail.cl}

\address{Daniel Sep\'ulveda \newline
Escuela de Matem\'aticas y Estad\'isticas, Universidad Central de Chile}
\email{daniel.sep.oe@gmail.cl}

\date{\today}
\thanks{}

\begin{abstract}
In this paper, 
ordinary and exponential dichotomies are 
defined in differential equations with equations with 
piecewise constant argument of general type. We prove
the asymptotic equivalence between the bounded solutions
of a linear system and a perturbed system
with integrable and bounded perturbations.
\end{abstract}
\maketitle

\numberwithin{equation}{section}
\newtheorem{thm}{Theorem}[section]
\newtheorem{lem}[thm]{Lemma}
\newtheorem{prop}[thm]{Proposition}
\newtheorem{cor}[thm]{Corollary}
\newtheorem{defn}{Definition}[section]
\newtheorem{conj}{Conjecture}[section]
\newtheorem{exam}{Example}[section]
\newtheorem{rem}{Remark}[section]
\allowdisplaybreaks


\section{Introduction}

The study of differential equations with piecewise constant argument  is motivated 
by several applications coming from different fields of science 
and by their own mathematical definition as  hybrid dynamical systems
\cite{2,Busenberg2,24}. For a longer discussion on 
applications  consult the references \cite{5,9,16,17,21,27,28}.
Here the meaning of hybrid is given in the sense that they combine the behavior
of differential and difference equations. 
In general, the typical 
form of this kind of equations is given by 
the following functional equation
\begin{eqnarray}
 \mathbf{x}'(t)=\mathbf{F}(t,\mathbf{x}(t),\mathbf{x}(\gamma(t))),
 \label{eq:eq_generica}
\end{eqnarray}
where $\mathbf{x}:\R\to \C^p$  is the unknown function,
$t\in\R$ usually denotes the time, $\mathbf{F}$
is a given function from $\R\times\C^p\times\C^p$ to $\C^p$,
and $\gamma$ is a given general step function 
in the sense that
\begin{eqnarray}
\left.
\begin{array}{l}
\mbox{$\gamma:\R\to\R$ is defined by $\gamma(t)=\zeta_i$ for $t\in I_i=[t_i,t_{i+1})$,
 where $\{t_i\}_{i\in \mathbb{Z}}$ and $\{\zeta_i\}_{i\in \mathbb{Z}}$}\\
 \mbox{are two given (fix) sequences such that $t_i\leq\zeta_i\leq t_{i+1}$
 with $t_i<t_{i+1}$ for each $i\in \mathbb{Z}$}\\
 \mbox{and $t_i\rightarrow \pm\infty $ when $i\rightarrow \pm\infty$.}
\end{array}
\right\}
\label{eq:genstepf}
\end{eqnarray}
The genesis of the study of this kind of functional equations 
goes back to the work of Myshkis~\cite{22}, who  proposed 
an equation of type \eqref{eq:eq_generica}-\eqref{eq:genstepf}
with  the  particular step functions
$\gamma(t)=[t]$ and $\gamma(t)=2\left[ (t+1)/2 \right]$.
Here $[\cdot]$ denotes the greatest
integer function. By simplicity of the presentation, we use  
hereinafter the terms DEPCA 
and DEPCAG to refer the 
differential equations with piecewise constant argument
when the step function is based on the greatest
integer function  and when the step function is of the general type given
by \eqref{eq:genstepf}, respectively.
In particular, note that the equations studied by  Myshkis
are DEPCAs.
Later, in the early 80's,
a systemic analysis of \eqref{eq:eq_generica}-\eqref{eq:genstepf} 
was introduced
by Wiener and collaborators,
see \cite{1,11,29,30} and references therein.
Afterwards, the contribution to the development of 
the theory was given by many authors see for instance
\cite{3,4,5,6,9,12,16,17,18,20,24,25,27,29,30,31,32,33,34}.
Nowadays, there exist 
an intense and increasing interest to understand the qualitative 
behavior, to get novel applications and to solve numerically
the equations, since a general theory
for \eqref{eq:eq_generica}-\eqref{eq:genstepf} is far to be closed,
see \cite{5,akhmet_Bneural,castilloEJDE,chavez0,chavez,10,cnn_chiu,26,Veloz}.

An important point to observe  is the notion of the solution
or more generally the types of approach to analyze  
\eqref{eq:eq_generica}-\eqref{eq:genstepf}. 
Actually, generally speaking,
the notion of solution for functional differential 
equations is one of the most important tasks.
Now, we recall that the original notion of 
the integration (or solution) of a DEPCA was introduced 
in \cite{1,11,30} and 
is based on 
the reduction to discrete equations. This approach presents
some disadvantages when we require a generalization to 
analyze DEPCAGs. 
In particular, for instance to solve the Cauchy problem
requires that the initial moments must be integers, 
see  \cite{30} for details. Another approach to study
a general quasilinear DEPCAG was introduced
by  Akhmet \cite{3,4}, and is based 
on the construction of an equivalent  integral equation
and remarking the clear influence of the discrete part. 
The methodology of Akhmet permits to overpass the 
difficulties of the methodology of Wiener
and collaborators. Moreover, 
Akhmet adapt the notion of solution given by Wiener
and used previously by Papaschinopoulos 
to study a particular type of DEPCA, see \cite{24}.
The notion of Akhmet solution is given in terms
of continuity of the solution on each $t_i$,
the existence of the derivatives on each $t$ with
possible exception of some $t_i$ and the local
satisfaction of the equation, see definition~\ref{def_solution}
below. Then, 
in spite of  its functional character
a quasilinear DEPCAG has
similar properties to ordinary differential equations.
For further details, consult for instance \cite{7,15,16,18,19,26}. 
Here, in this paper, we use the approach of Akhmet.
Thus, we do not need to impose any restrictions on
the discrete equations
and we assume more easily verifiable conditions on the
coefficients,
similar to those for ordinary differential equations.


In this paper, we are interested in the
asymptotic equivalence of some DEPCAGs. Now,
in order to precise the different type of systems
which will be used in the paper,
we introduce a particular notation of each case.
Indeed, throughout  the paper,
we consider that  $x,z,u,y,w,v$ satisfies
the following particular cases of \eqref{eq:eq_generica}-\eqref{eq:genstepf}:
\begin{eqnarray}
x'(t)&=&A(t)x(t),\label{A}\\
z'(t)&=&A(t)z(t)+B(t)z(\gamma(t)),\label{depca_lineal}\\
u'(t)&=&B(t)u(\gamma(t))\label{B}
\\
y'(t)&=& A(t)y(t)+B(t)y(\gamma(t))+g(t),\label{depca_lineal+g}\\
w'(t)&=& A(t)w(t)+B(t)w(\gamma(t))+f(t,w(t),w(\gamma(t))). 
\label{depca_lineal+f}
\\
v'(t)& =& A(t)v(t)+B(t)v(\gamma(t))+g(t)+f(t,v(t),v(\gamma(t))).
\label{depca_g+f}
\end{eqnarray}
Note that \eqref{A}-\eqref{depca_lineal+g} are linear and 
\eqref{depca_lineal+f}-\eqref{depca_g+f}
are  nonlinear. The specific  hypotheses 
about the different functions given on 
\eqref{A}-\eqref{depca_g+f}
are summarized
on subsection~\ref{sect:assump}.

We have found some previous results on the literature
and specifically focused on the analysis DEPCAs of 
DEPCAGs of types \eqref{depca_lineal}-\eqref{depca_lineal+f}.
Particularly, here we comment the works of Akhmet~\cite{3,4}
and Pinto~\cite{cauchy_and_green}, since they are more
close to our contributions. Indeed,
Akhmet \cite{3,4}, by applying, his approach
has obtained fundamental
results about the variation of constants 
formula and the stability
of the perturbed system \eqref{depca_lineal+f}. 
Now, Pinto  in \cite{cauchy_and_green}, by applying
a combined methodology based on the 
Green matrix and Akhmet approaches,
has proven 
some important results
related to the analysis of the DEPCAGs \eqref{depca_lineal}-\eqref{depca_lineal+f}.
In particular, he defines an appropriate Green matrix
associated to  \eqref{depca_lineal}, then formulated the solution of 
\eqref{depca_lineal+g} in terms of this Green matrix and 
subsequently characterize the solution of  \eqref{depca_lineal+f}
by an integral equation of the first kind.  Using the  
integral equation and a Gronwall type inequality for DEPCAGs,
he deduces an existence and uniqueness of solutions for 
\eqref{depca_lineal+g}.
Moreover, using this approach he guarantees that the
zero solution of \eqref{depca_lineal} is exponentially asymptotically 
stable. Finally,  he  gives some equivalence results on stability
and compare his results with the corresponding ones of Akhmet~\cite{3,4}.
Probably, the  major advantages of the approach introduced
in \cite{cauchy_and_green} are two. Firstly, permits the introduction 
of more general hypotheses on the coefficients.
Second, the Green matrix  type taking account of the decomposition 
of any  interval, $I_i=[t_i,t_{i+1})$
in the advanced 
intervals $I_i^{+}=[t_i,\zeta_i]$ and the delayed intervals
$I_i^{-}=[\zeta_i,t_{i+1})$, hence permits to 
analyze the alternately advanced and
delayed differential systems in a unified way.

The paper is organized as follows. In section~\ref{sec:preli}, 
we introduce several preliminary concepts
and results like the hypotheses,
the Cauchy and Green matrices type, the notion of solution
for DEPCAGs, the Gronwall type inequality for DEPCAGs
and the notion of dichotomies and stability.
Moreover, in section~\ref{sec:preli} we give a detailed example
for \eqref{depca_lineal} with $A$ and $B$ constants.
In sections~\ref{sec:dos}-\ref{sec:cuatro}, we present the main results of the paper
which are summarized as follows
\begin{enumerate}[(1)]
 \item {\it Stability of solutions for \eqref{depca_lineal+f}.}
 In Theorem~\ref{thm:Stabildepca_linel} we prove that 
 the $\sigma$-exponentially stability of  
 the zero solution for the linear DEPCAG \eqref{depca_lineal}
 implies  $\sigma_0$-exponentially stability of
 the zero solution for \eqref{depca_lineal+f},
 where $\sigma_0$ is defined in terms of $\sigma,A,f,\{t_i\}_{i\in\Z}$ and 
 $\{\zeta_i\}_{i\in\Z}$.
 
 \item {\it Bounded solutions for \eqref{depca_lineal+g}.} In 
 Proposition~\ref{prop_45} we prove that 
 $\sigma$-exponentially stability of \eqref{depca_lineal+g}
 and a convergence of a series given in terms of the
 fundamental solutions for \eqref{A} and \eqref{depca_lineal}
 implies that equation \eqref{depca_lineal+g} has a unique bounded 
 solution on  $\R$. Moreover, in Theorem~\ref{theo_sigma_exp_unique_sol}
 we prove that there exists  a unique bounded solution  
 of the non-homogeneous linear DEPCAG \eqref{depca_lineal+g}
 by assuming the linear DEPCAG \eqref{depca_lineal} 
 has a $\sigma$-exponential dichotomy and 
 two series in terms of the
 fundamental solutions for \eqref{A} and \eqref{depca_lineal}
 and the  associated projection to the dichotomy converges,
 see \cite{akhmet_Bneural} for other results. In \cite{akhmet_Bneural},
 periodic and almost periodic solutions are also studied.

 \item {\it Asymptotic equivalence of \eqref{depca_lineal+g} and \eqref{depca_g+f}}.
 In Theorem~\ref{thm_L_fintegrable}, we prove that if the linear system 
 \eqref{depca_lineal} has an ordinary dichotomy
 and in \eqref{depca_g+f} $f$ is integrable, 
 then  there exists a homeomorphism
between the bounded solutions of \eqref{depca_lineal+g} and
the bounded solutions of \eqref{depca_g+f}.

 \item {\it Bounded solutions for \eqref{depca_lineal+f}}.
 In Theorem~\ref{teo:dsolgendecag} we prove that:
  if \eqref{depca_lineal} has a $\sigma$-exponential dichotomy 
  satisfying the requirements of Theorem~\ref{theo_sigma_exp_unique_sol},
  then for any $\xi\in P\C^{p}$ 
the nonlinear equation \eqref{depca_lineal+f} has
a unique bounded solution $w$ on $[t_0,\infty)$,
with  $Pw(t_0)=\xi$, the correspondence $\xi\to w$
is continuous and the bounded solution of equation \eqref{depca_lineal+f}
  converges exponentially to zero as $t\to\infty$.
\end{enumerate}
Moreover, we  prove introduce four corollaries with
particular importance in Corollary~\ref{eq:homeo_cp}
we deduce that here exists a homeomorphism between $\C^p$ and 
the bounded solutions of \eqref{B}.

\section{Preliminaries}
\label{sec:preli}

\subsection{General and stability assumptions and notation} 
\label{sect:assump}
$ $

In this section we summarize several hypotheses used thorough of the
paper. We organize these assumptions in two big groups
(H1)-(H4) and (S1)-(S3). The distinction obeys fundamentally to 
the fact that (S1)-(S3) are frequently 
needed by the exponential stability
results. 

The first group is given as follows
\begin{enumerate}
\item[(H1)] Let us denote by $\C^{n\times m}$ the vectorial space 
of complex matrices of size $n\times m$. We assume that the coefficients
of \eqref{A}-\eqref{depca_lineal+f} defined by the functions
$A,B:\R\to \C^{p\times p}$ and $g:\R\to \C^{p}$ 
are locally integrable in $\mathbb{R}$. 

\item[(H2)] For an arbitrary matrix valued function $M \in L^1_{\rm loc}(\R;\C^{p\times p})$
and for each $i\in\mathbb{Z}$ consider the following notation
$\rho_i(M)=\rho_i^{+}(M)\rho_i^{-}(M)$ with 
$\rho_i^{\pm}(M)=\exp\left(\int_{I^{\pm}_i}\vert M(s)\vert ds \right),$
where $\{I^+_i,I^-_i\}$ is a partition of $I_i$ defined as follows
\begin{eqnarray}
\mbox{$I^+_i=[t_i,\gamma(t_i)]=[t_i,\zeta_i]$ and 
$ I^{-}_i=(\gamma(t_i),t_{i+1})=(\zeta_i,t_{i+1})$.}
\label{eq:adv_del_int}
\end{eqnarray}
The sets $I^+_i$ and $I^-_i$ are so called the advanced and delayed
intervals, respectively. We suppose that 
the functions $A$ and $B$ on equations \eqref{A}-\eqref{depca_lineal+f}
satisfy the following relations 
\begin{eqnarray}
\rho(A)=\sup_{i\in \mathbb{Z}}\rho_i(A)<\infty,\ \
\nu_i^{+}(B)\leq \nu^{+}<1,\ \nu_i^{-}(B)\leq \nu^{-}<1,
\label{rho_A}
\end{eqnarray}
where $\nu_{i}^{\pm}=\rho_{i}^{\pm}(A) \ln (\rho_i^{\pm}(B)) $.

\item[(H3)] We assume that the function $f$
given as the third term on the right hand of the equation
\eqref{depca_lineal+f} satisfies the following three requirements:
(i) $f:\mathbb{R}\times \mathbb{C}^p\times \mathbb{C}^p\to\mathbb{C}^p $
is a continuous function, i.e. 
$f\in C(\mathbb{R}\times \mathbb{C}^p\times \mathbb{C}^p,\mathbb{C}^p)$;
(ii) $f(t,0,0)=0$ for each $t\in \mathbb{R}$; and
(iii) there exists $\eta\in L^1_{\rm loc}(\R;[0,\infty))$
such for all $t\in\mathbb{R}$ and all $\ x_i,y_i
\in\mathbb{C}^p,i=1,2,$  the inequality
\begin{eqnarray}
\Big\vert f(t,x_1,y_1)-f(t,x_2,y_2) \Big\vert
\leq \eta(t)\Big(\vert x_1-x_2 \vert +\vert y_1-y_2 \vert \Big)
\end{eqnarray}
holds.

\item[(H4)] Let $X(t)$ a fundamental matrix of solutions for \eqref{A}
and denote by $\Phi$ the also called fundamental matrix of \eqref{A}
and defined by $\Phi(t,s)=X(t)X^{-1}(s)$  for $(t,s)\in \R^2$. Consider
that $J:\R^2\to \C^{p\times p}$ is defined as follows
\begin{eqnarray}
J(t,\tau)=I_p+\int_{\tau}^{t}\Phi(\tau,s)B(s)ds,
\quad
\mbox{where $I_p$ is the $p\times p$ identity matrix}.
\end{eqnarray}
For each $i\in\Z$, we assume that 
the matrix $J(t,\tau)$ is non-singular
for all $t,\tau\in [t_i,t_{i+1}) $. 
\end{enumerate}
Here we introduce three comments about  (H1)-(H4): 
{\it (a)} note that when $A$ and $B$ are constant, the relation \eqref{rho_A}
is strongly simplified and naturally a simpler condition can be obtained;
{\it (b)} the matrix $\Phi$ defined on (H4) satisfy the following properties:
\begin{eqnarray}
\Phi(\tau,\tau)=I_p,\quad
\Phi(t,s)\Phi(s,\tau)=\Phi(t,\tau),\quad
\Phi(t,s)=\Phi^{-1}(s,t),\quad
\forall (s,t,\tau)\in\R^3;
\label{eq:propPhi}
\end{eqnarray}
and {\it (c)} (H2) implies (H4) (see Lemma~\ref{lema_cotas}).

Now, we introduce the following second group of hypotheses
\begin{enumerate}
\item[(S1)] For a given $t\in\R$
denote by $i(t)$ the unique integer such that $t\in I_{i(t)}$.
We assume that estimate
\begin{eqnarray*}
\sup_{i(t)\in\mathbb{Z}}\;\sup_{t\in I_{i(t)}}\vert Z(t,t_{i(t)})\vert <\infty
\end{eqnarray*}
is satisfied by  $Z(t,s)$ the fundamental solution of \eqref{depca_lineal}
(see subsection~\ref{subsec:FSdepca_lineal}).

\item[(S2)] $\displaystyle\inf_{i\in\Z}\{t_{i+1}-t_{i}\}>0$

\item[(S3)] The lengths of the $I^\pm_i$ satisfies the following
bound
\begin{eqnarray}
\overline{t}=\sup_{i\in\mathbb{Z}}\max
\Big\{\zeta_i-t_i,t_{i+1}-\zeta_i\Big\}<\infty
\label{eq:deftline}
\end{eqnarray}
\end{enumerate}
We note four facts: {\it (a)} (S1) is more general than (H2),
since (H2) implies (S1) and condition (S2);
{\it (a)} (S2) implies (S3);  {\it (c)} 
for $A$ and $B$ constant matrices, 
the conditions (S2) and (S3) imply (S1);
and {\it (d)} 
nowadays the most studied DEPCAGs 
satisfy (S3), see for instance \cite{1,3,6,11,12,30}.

\subsection{Notion of solution for DEPCAGs} $ $

We state the notion of solutions for DEPCAG by 
following the given \cite{1,10,11,12,30}. More precisely 
we have the following definition.

\begin{defn}
\label{def_solution}
Consider a continuous function $F:I\times\C^p\times\C^p\to\times\C^p$
where $I\subset\R$ and $\gamma$ is a general step function 
in the sense of \eqref{eq:genstepf}. Then a function $\mathbf{v}:I\to \C^p$ such that
$t\mapsto \mathbf{v}(t)$ is called a solution of the following DEPCAG
\begin{eqnarray}
 \mathbf{v}'(t)=F(t,\mathbf{v}(t),\mathbf{v}(\gamma(t)))
 \label{eq:depacag_gnl}
\end{eqnarray}
on an interval $I$ if:
\begin{enumerate}[(a)]
\item $\mathbf{v}$ is continuous on  $I$;

\item the derivative $\mathbf{v}'(t)$ exists at each point $t\in I$ 
with the possible exception of the points $t_i \in I,i\in \mathbb{Z}$, 
where the one sided derivatives exist; and

\item $\mathbf{v}$ satisfies pointwise
the equation \eqref{eq:depacag_gnl} on each interval 
$(t_i,t_{i+1})\subset I, i\in\mathbb{Z}$, 
and \eqref{eq:depacag_gnl} holds for the right derivative of $\mathbf{v}$ 
at the points $t_i\in I, i \in \mathbb{Z}$.
\end{enumerate}
\end{defn}

Naturally, the  definition~\ref{def_solution} can be applied to precise the  
notion of solution for the DEPCAG of types 
\eqref{depca_lineal}-\eqref{depca_lineal+f} by considering that 
$F$ has a particular definition in each case.

\subsection{The fundamental matrix of solutions of \eqref{depca_lineal}}
\label{subsec:FSdepca_lineal}
$ $

In several parts of the paper we need  
the fundamental matrix of solutions of \eqref{depca_lineal}.
For instance, to define the Green matrices 
the fundamental matrix of solutions
of DEPCAG \eqref{depca_lineal} is a central concept. Now,
we recall that, under the
assumption (H4), there exists 
$Z(t,s)\in \C^{p\times p}$ a fundamental matrix of solutions
of DEPCAG \eqref{depca_lineal}, see for instance \cite{4,27}.
Indeed, to construct $Z$ we proceed in two steps.
Firstly, we define $Z(t,\tau)$ for $t\ge \tau$ and then for $t\le \tau$.
Let us consider that $(t,\tau)\in [t_j,t_{j+1})\times [t_i,t_{i+1})$ with 
$t>\tau$ and $j\ge i$.
By induction we can deduce that
\begin{eqnarray}
Z(t,\tau)
&=&E(t,\gamma(t_j)) E(t_j ,\gamma(t_j))^{-1} \nonumber\\
& & \times  \prod_{k=j}^{i+2} 
\Big(E(t_k,\gamma(t_{k-1}))E(t_{k-1},\gamma(t_{k-1}))^{-1}\Big)
\; E(t_{i+1},\gamma(\tau))E(\tau,\gamma(\tau))^{-1},
\nonumber\\
&=& E(t,\zeta_j) E(t_j ,\zeta_j)^{-1}
\nonumber\\
&&\times\prod_{k=j}^{i+2} 
\Big(E(t_k,\zeta_{k-1})E(t_{k-1},\zeta_{k-1})^{-1}\Big)
 E(t_{i+1},\zeta_i)E(\tau,\zeta_i)^{-1},
\quad\mbox{for } t\ge \tau,
\label{Z_E}
\end{eqnarray}
where $\prod$ denotes the backward product
$\displaystyle\prod_{k=m}^{n}C_{k}=C_{n}C_{n-1}\cdots C_{m}$ 
and $E$ is the matrix defined as follows 
\begin{eqnarray}
 E(t,\tau)=\Phi(t,\tau) J(t,\tau) 
 \quad
 \mbox{with $\Phi$ and $J$ defined on (H4).}
 \label{eq:def_of_E}
\end{eqnarray}
Note that the property of no-singularity of the matrix $J(\cdot,\zeta_i)$ on the
interval $I_i$ is the minimal requirement for the well definition of $Z$ given
on \eqref{Z_E}. 
Now, to define $Z(t,\tau)$
for $t<\tau$ we note that
$Z(t,\tau)$ defined by \eqref{Z_E} satisfies similar properties to
\eqref{eq:propPhi}, since we can easily deduce that
\begin{eqnarray}
Z(\tau,\tau)&=&I_p,\hspace{1.9cm}\mbox{for $\tau\in\R$,}\\
Z(t,s)Z(s,\tau)&=&Z(t,\tau), \hspace{1.2cm}\mbox{for } \tau\geq s\geq t, 
\label{Z_multiplicative}\\
Z(t,\tau)&=&[Z(\tau,t)]^{-1},
\qquad\mbox{for } t\le \tau.
\label{Z_inverse}
\end{eqnarray}
In particular, the property \eqref{Z_inverse}  is important for our purpose 
since it allows to define $Z(t,\tau)$ for $ t<\tau$ using \eqref{Z_E}.
Hence $Z(t,s)$  is completely defined  on $\mathbb{R}^2$ 
by \eqref{Z_E} and \eqref{Z_inverse} and naturally
\begin{eqnarray}
\mbox{$z(t)=Z(t,\tau)z(\tau)$ for $t\ge \tau$ or 
$z(t)=[Z(t,\tau)]^{-1}z(\tau)$ for $t<\tau$}
\label{solution}
\end{eqnarray}
defines the solution of \eqref{depca_lineal} with initial condition 
$(\tau,z(\tau)).$

From \eqref{Z_E} and \eqref{Z_inverse} and for notational convenience,
we introduce the matrices $Z^{\pm}$ as follows
\begin{eqnarray}
\mbox{$Z^+(t,s)=Z(t,s)$ for $t\ge s$ and 
$Z^-(t,s)=[Z(s,t)]^{-1}$ for $t<s$,}
\label{eq:Z_mas_menos_notation}
\end{eqnarray}
where  $Z$ is evaluated by \eqref{Z_E}.

The results of this subsection are formalized below on Lemma~\ref{lema_cotas}
and Corollary~\ref{cor:depca_lineal}.

\subsection{Green matrices type}
$ $

In this subsection, we recall the concepts and notation of  Green
matrices $G_k$ and $G$.

\begin{defn}
\label{def:green_matrixGk}
(Green matrix $G_k(t,s)$ for $(t,s)\in [t_k,t_{k+1}]^2$)
Consider the notation defined on \eqref{eq:genstepf}.
For a given $k\in\Z$, the Green matrix type $G_k$ is defined from
 $[t_k,t_{k+1}]\times [t_k,t_{k+1}]$ to $\C^{p\times p}$ by the following relation
\begin{eqnarray}
G_{k}(t,s)&=&\begin{cases}
G_{k}^{+}(t,s), & \mbox{ for } (t,s)\in [t_k,t_{k+1}]\times [t_k,\gamma(s)],\\
G_{k}^{-}(t,s), & \mbox{ for } (t,s)\in [t_k,t_{k+1}]\times (\gamma(s),t_{k+1}], 
\label{green_matrix_type_local}
\end{cases}
\end{eqnarray}
where 
\begin{eqnarray*}
G_{k}^{+}(t,s)&=& Z^+(t,\tau)\Phi(\tau ,s),\;\;\mbox{ for } 
\tau\leq s\leq\gamma(s),\;\; t_{k}\leq \tau \leq t,\\
G_{k}^{-}(t,s)&=& Z^-(t,\tau)\Phi(\tau ,s),\;\;\mbox{ for } 
\gamma(s)< s \leq \tau,\;\; t \leq \tau \leq t_{k+1}.
\end{eqnarray*}
Here the notation $Z^{\pm}$ is defined on 
\eqref{eq:Z_mas_menos_notation}.
\end{defn}

In particular, 
the Green matrix \eqref{green_matrix_type_local},
for the advanced situation $t_k\leq s\leq \gamma(s)=t_{k+1}$
\begin{eqnarray*}
G_{k}^{+}(t,s)&=&\begin{cases}
Z^+(t,t_{k})\Phi(t_{k},s), & \mbox{ for }t_{k}\leq s\leq\gamma(t_k),\;\; s<t,\\
\Phi(t,s), & \mbox{ for }t\leq s\leq\gamma(t_k),
\end{cases}
\end{eqnarray*}
and for the delayed situations $t_k= \gamma(s)<s$
\begin{eqnarray*}
G_{k}^{-}(t,s)&=&\begin{cases}
Z^-(t,t_{k+1})\Phi(t_{k+1},s), & \mbox{ for }\gamma(s)<s\leq t_{k+1},\;\; t>s,\\
\Phi(t,s), & \mbox{ for }\gamma(s)<s\leq t<t_{k+1}.
\end{cases}
\end{eqnarray*}

\begin{defn}
\label{def:green_matrixG}
(Green matrix $G(t,s)$ for $(t,s)\in \R^2$)
Consider the notation  $i(t)$ given on (S1).
The Green matrix type $G:\R^2\to\C^{p\times p}$ 
for $s>t$ is defined as follows
\begin{eqnarray}
G(t,s)=
\begin{cases}
G_{i(t)}(t,s), & i(s)=i(t),
\\
G_{i(t)}(t_{i(t)+1},s)+G_{i(t)}(t,t_{i(t)+1}), & i(s)=i(t)+1,
\\
\displaystyle
 G_{i(t)}(t_{i(t)+1},t)
+\sum_{k=i(t)+1}^{i(s)-1}G_{k}(t_{k+1},t_k)
+G_{i(s)}(s,t_{i(s)-1}),
& i(s)>i(t)+1,
\end{cases}
\label{green_matrix_type_depcag}
\end{eqnarray}
and for $s<t$ by the following relation
\begin{eqnarray}
G(t,s)=
\begin{cases}
G_{i(s)}(t,s), & i(s)=i(t),
\\
G_{i(s)}(t_{i(s)+1},s)+G_{i(s)}(t,t_{i(s)+1}), & i(s)=i(t)+1,
\\
\displaystyle
 G_{i(s)}(t_{i(s)+1},s)
+\sum_{k=i(s)+1}^{i(t)-1}G_{k}(t_{k+1},t_k)
+G_{i(t)}(t,t_{i(t)-1}),
& i(t)>i(s)+1,
\end{cases}
\label{green_matrix_type_depcag:2}
\end{eqnarray}
where $G_k(\cdot,\cdot)$ is the matrix introduced 
on Definition~\ref{def:green_matrixGk}.
\end{defn}

\vspace{0.5cm}
We note that
\begin{eqnarray*}
G(t,s)&=&\begin{cases}
G^{+}(t,s), &\mbox{ for }s\leq\gamma(s),\\
G^{-}(t,s), &\mbox{ for }\gamma(s)<s,
\end{cases}
\end{eqnarray*}
where
$G^{\pm}(t,s)=\sum_{k=i(s)}^{i(t)}G_{k}^{\pm}(t,s)$.
This fact justifies the name of 'Green matrix type' for $G$.
Moreover, $G_{k}(t_{k+1},t_{k})=G_{k}^{+}(t_{k+1},t_{k})=Z(t_{k+1},t_k)$, 
which gives the recurrence relation:
\begin{eqnarray}
x(t_{i+1})=Z(t_{i+1},t_i)x(t_i), i\in \mathbb{Z}.\label{recurrence}
\end{eqnarray}
Note that from \eqref{recurrence}, we have
\begin{eqnarray}
X(t_{j+1})=\prod_{k=i}^{j}Z_{k}(t_{k+1},t_k)x(t_i),\ \ i\leq j,
\end{eqnarray}
which gives another way to obtain formula \eqref{Z_E}. 
It allows also to solve the linear non-homogeneous
DEPCAG \eqref{depca_lineal+g}.

Here we recall
an important result given in \cite{cauchy_and_green}.

\begin{lem}
\label{lema_cotas}
\cite{cauchy_and_green}
Assume that the condition (H2) is fulfilled,
then the condition (H4) holds and the matrices $Z(t,s)$ and 
$Z(t,s)^{-1}$ are well defined for any
$(t,s)\in\R^2$ with $t\ge s$. 
Moreover, there exists a positive
constant number $\alpha$ such that
\begin{eqnarray}
\vert \Phi(t,s) \vert\leq \rho(A), 
\quad
\vert Z(t,s)\vert\leq  \alpha
\quad\mbox{and}\quad
\vert  G_i(t,s)\vert\leq \alpha\rho(A) 
\mbox{ for }(t,s)\in [t_i,t_{i+1}],
\label{cotas_Z_Phi_G}
\end{eqnarray}
for each $i\in\mathbb{Z}$.
\end{lem}

\subsection{Variation of parameters formulas}
$ $

A variation of parameters formula to
\eqref{depca_lineal+g} can be deduced
by assuming that (H1) and (H4)  holds. 
Indeed, 
in \cite{cauchy_and_green} was proven that the solution of the equation
of \eqref{depca_lineal+g} is given by
\begin{eqnarray}
y(t)
&=&Z(t,\tau)y(\tau) +Z(t,\tau)
\int_{\tau}^{\zeta_i} \Phi(\tau,s)g(s)ds
+\sum_{k=i+1}^{j}Z(t,t_{k})
\int_{t_k}^{\zeta_k}\Phi(t_k,s)g(s)ds 
\nonumber\\
& &+\sum_{k=i}^{j-1}Z(t,t_{k+1})
\int_{\zeta_k}^{t_{k+1}}\Phi(t_{k+1},s)g(s)ds
+\int_{\zeta_j}^{t}\Phi(t,s)g(s)ds,
\quad
\mbox{for $(\tau,t)\in I_i\times I_j$.}
\qquad
\label{var_parametro_Z}
\end{eqnarray}
In particular, 
we have that 
\begin{eqnarray}
y(t)
&=&Z(t,\tau)y(\tau) +\int_{\tau}^{\zeta_i} G_i^{+}(t,s)g(s)ds
+\sum_{k=i+1}^{j}\int_{t_k}^{\zeta_k}G_{k}^{+}(t,s)g(s)ds
\nonumber\\
& &+\sum_{k=i}^{j-1}\int_{\zeta_k}^{t_{k+1}}G_{k}^{-}(t,s)g(s)ds
+\int_{\zeta_j}^{t}G_{j}^{-}(t,s)g(s)ds,
\quad
\mbox{for $(\tau,t)\in I_{i}^{+}\times I_j^{-}$,}
\label{var_parametro_G}
\end{eqnarray}
where the notation $I_{i}^{\pm}$ is defined in \eqref{eq:adv_del_int}.
Note that,
using Definition~\ref{def_solution}, a similar 
formula can be obtained for $(\tau,t)\in I_i^{-}\times I_{j}^{-} $ .
To be more precise, the following theorem  can be stated.

\begin{thm}
\label{theo_solution}
\cite{cauchy_and_green}
Assume that the hypotheses (H1) and (H4) (or (H1) and (H2)) are 
fulfilled. Then, for any $(\tau,\xi)\in\R\times\C^{p}$
the solution of \eqref{depca_lineal+g} such that $y(\tau)=\xi$
is defined on $\R$ and is given by 
\begin{eqnarray}
y(t)=Z(t,\tau)\xi +\int_{\tau}^{t}G(t,s)g(s)ds.  
\label{theo_var_parametro_depca}
\end{eqnarray}
In particular the formula \eqref{theo_var_parametro_depca} 
is reduced  to \eqref{var_parametro_Z} or 
\eqref{var_parametro_G} depending if $(\tau,t)\in I_i\times I_{j}$ 
or $(\tau,t)\in I_i^{+}\times I_{j}^{-}$, respectively.
\end{thm}

By application of Theorem \ref{theo_solution} to the DEPCAGs
\eqref{depca_lineal},
\eqref{B}, and
\eqref{depca_lineal+f}
we can deduce some useful results which are formalized in the 
following three corollaries.

\begin{cor}
\label{cor:depca_lineal}
\cite{cauchy_and_green}
Assume that $A(t)$ and $B(t)$ satisfies the requirements of hypothesis (H1).
Moreover, assume that 
(H4) or (H2)  are fulfilled. 
Then, the following assertions are valid 
\begin{enumerate}[(i)]
 \item there exists $Z:\R^2\to\C^{p\times p}$ 
 the fundamental solution matrix of the linear system 
\eqref{depca_lineal} and is given by \eqref{Z_E} and \eqref{Z_inverse},
\item for every $(\tau,\xi)\in\R\times\C^{p}$, 
there exists on all of $\mathbb{R}$ a unique solution \eqref{depca_lineal}
such that $z(\tau)=\xi$. This solution is given by \eqref{solution}.
\end{enumerate}
Furthermore, conversely, the existence of 
a solution $z(t)=z(t,\tau,\xi)$ of \eqref{depca_lineal}
defined on all of $\mathbb{R}$
implies that the condition (H4) must be true.
\end{cor}

\begin{cor}
\label{3.3}
\cite{cauchy_and_green}
Assume that $B(t)$ satisfies the requirement of hypothesis (H1).
Then, for every $(\tau,\xi)\in\R\times\C^{p}$
the solution of DEPCAG \eqref{B} 
such that $u(\tau)=\xi$
is given by the formulae \eqref{var_parametro_Z} and 
\eqref{theo_var_parametro_depca} 
with $u(t)=z(t)$, $g=0$,
$\Phi(t,s)\equiv I_p$, 
or equivalently  $u(t)=z(t)$, $g=0$,
$E(t,s)=J(t,s)=I_p+\int_{s}^{t}B(r)dr$,
\begin{eqnarray*}
Z(t,\tau)
&=& J(t,\gamma(t_j))J(t_j,\gamma(t_j))^{-1}\\
& & \prod_{k=i+2}^{j}\Big(J(t_{k},\gamma(t_{k-1}))
J(t_{k-1},\gamma(t_{k-1}))^{-1}\Big)\quad 
J(t_{i+1},\gamma(\tau))J(\tau,\gamma(\tau))^{-1},
\end{eqnarray*}
for $(t,\tau)\in [t_j,t_{j+1})\times [t_i,t_{i+1})$ with $t\ge\tau$.
\end{cor}

\begin{cor}
\label{cor_integral_equation}
\cite{cauchy_and_green}
Assume that the hypotheses (H1) and (H4) (or (H1) and (H2)) are 
fulfilled, 
then for any $(\tau,\xi)\in\mathbb{R}\times\mathbb{C}^{p}$, 
every  $w(t)=w(t,\tau,\xi)$  
solution of the quasilinear DEPCAG \eqref{depca_lineal+f}
such that $w(\tau)=\xi$,
satisfies the integral equation 
\eqref{theo_var_parametro_depca} with 
$g(s)=f(s,z(s),z(\gamma(s)))$ or
\begin{eqnarray}
w(t)=Z(t,\tau)\xi +\int_{\tau}^{t}G(t,s)f(s,w(s),w(\gamma(s)))ds. 
\label{integral_equation}
\end{eqnarray}
Conversely, any solution of the integral equation
\eqref{integral_equation}, is a solution in the sense of 
Definition~\ref{def_solution}, of the 
quasilinear DEPCAG \eqref{depca_lineal+f}.
\end{cor}

To close the subsection we remark that,
using the classical method of Wiener \cite{30}, (pp. 8,18,52,88), 
several authors
have obtained variation of parameters formula for the particular
cases of the steps functions $\gamma(t)$ given by $[t],[t+1/2],[t+1]$. 
They have obtained the compact formula:
\begin{eqnarray}
y(t)
&=&Z(t,\tau)y(\tau) +Z(t,\tau)\int_{\tau}^{\gamma(t_{i})} 
\Phi(\tau,s)g(s)ds \label{var_parametro_Z_}\\
& &+\sum_{k=i}^{j-1}Z(t,t_{k+1})
\int^{\gamma(t_{k+1})}_{\gamma(t_{k})}\Phi(t_{k+1},s)g(s)ds
+\int_{\gamma(t_{j})}^{t}\Phi(t,s)g(s)ds,\nonumber
\end{eqnarray}
for $(\tau,t)\in [t_i,t_{i+1})\times [t_j,t_{j+1})$, 
which follows from \eqref{theo_var_parametro_depca}, 
\eqref{green_matrix_type_depcag}-\eqref{green_matrix_type_depcag:2} and 
\eqref{green_matrix_type_local} 
since $G(t,s)=Z(t,t_{k+1})\Phi(t_{k+1},s)$ for
$s\in [\gamma(t_k),\gamma(t_{k+1})]=[\zeta_k,\zeta_{k+1}]$. 
Under other conditions, this formula was
extended to any DEPCAG by Akhmet \cite{4}.
This formula takes account of the intervals 
$[\zeta_k,\zeta_{k+1}]$ instead of the advanced intervals
$I_k^{+}$ and the delayed intervals 
$I_k^{-}$. This representation 
allows not to see the Green matrix type.

\subsection{Gronwall type inequalities for DEPCAGs}
$ $

The Gronwall type inequality for DEPCAG introduced and proved
in \cite{25} is formulated in 
the following lemma.

\begin{lem}
\label{lema_gronwal}
\cite{25}
Let $u,\eta:\mathbb{R}\rightarrow [0,\infty)$ be two functions
such that $u$ is continuous and $\eta$ is locally integrable satisfying
\begin{eqnarray}
\theta=\sup_{i\in\mathbb{Z}}\Big\{\;
\theta_i\;:\; \theta_i:=2\int_{I_i}\eta(s)ds
\,\Big\}<1.
\label{eq:defteta_gronwall}
\end{eqnarray}
Suppose that for $\tau\leq t$ or $t\leq \tau$, 
we have the inequality
\begin{eqnarray*}
u(t)\leq u(\tau)+\left\vert  
\int_{\tau}^{t} \eta(s)[u(s)+u(\gamma(s))]ds \right\vert.
\end{eqnarray*}
Then
\begin{eqnarray*}
u(t)&\leq & u(\tau) \exp \left\{ \tilde{\theta} 
\int_{\tau}^{t}\eta(s)ds \right\}
\quad
\mbox{with}
\quad
\tilde{\theta}
=\frac{2-\theta}{1-\theta}.
\end{eqnarray*}
\end{lem}

If we consider the forward and backward 
situation in a separated way, then in \eqref{eq:defteta_gronwall}
instead of integration on the all $I_i$ to define $\theta_i$
we need only an integration on $I^+_i$ or $I^-_i$
respectively. More precisely, we have the following result.

\begin{cor}\cite{cauchy_and_green}
The results in Lemma \ref{lema_gronwal} are true
\begin{eqnarray*}
&&\mbox{for}\quad t\geq\tau,\quad\mbox{if}\quad
\theta=\sup_{i\in\mathbb{Z}}\Big\{\;
\theta^+_i\;:\; \theta^+_i:=2\int_{I^+_i}\eta(s)ds
\,\Big\}<1,\\
&&\mbox{for}\quad t\leq\tau,\quad\mbox{if}\quad
\theta=\sup_{i\in\mathbb{Z}}\Big\{\;
\theta^-_i\;:\; \theta^-_i:=2\int_{I^-_i}\eta(s)ds
\,\Big\}<1.
\end{eqnarray*}
\end{cor}

\subsection{Existence and uniqueness of solution of the quasilinear system
\eqref{depca_lineal+f}} $ $

The existence, uniqueness, boundedness, 
stability and continuous dependences of
the solutions $w(t)=w(t,\tau,\xi)$ of 
DEPCAG \eqref{depca_lineal+f} are precisely  stated
as follows.

\begin{prop}
\label{prop_stability}
\cite{cauchy_and_green}
Assume that conditions (H1), (H3) and (H4) 
(or (H1)-(H3))are fulfilled. Moreover,
assume that the given function $A$ in (H1) and the existing function
$\eta$ in (H3) satisfies the estimate $\alpha\rho(A)\theta<1$
with $\rho(A),$ $\alpha$ and $\theta$   given on
\eqref{rho_A},
\eqref{cotas_Z_Phi_G} and \eqref{eq:defteta_gronwall}, respectively.
Then, for every $(\tau,\xi)\in\mathbb{R}\times\mathbb{C}^p$, there exists 
$w(t)=w(t,\tau,\xi)$ solution of \eqref{depca_lineal+f} in the sense
of Definition~\ref{def_solution} and
satisfying the following properties:
\begin{enumerate}[(i)]
\item  $w(\tau)=\xi$,
\item $w$ is defined on all of $\mathbb{R}$,
\item $w$ is solution of the integral equation \eqref{integral_equation},
\item $w$ is unique and depends continuously on $\tau$ and $\xi$.
\end{enumerate}
Moreover, if there exists a constant $c\geq 1$  such that
\begin{eqnarray}
\vert Z(t,s)\vert\leq c,\mbox{ for }t\geq s 
\label{1_condition_stability}
\end{eqnarray}
and if $\eta\in L^{1}(\mathbb{R})$, then $w$ is bounded 
and is stable, namely
\begin{eqnarray}
\exists c_1\in \R^+\quad :\quad 
\vert w(t,\tau,\xi_1)-w(t,\tau,\xi_2)\vert\leq c_1
\vert \xi_1 -\xi_2 \vert,\quad  
\forall t\geq \tau, \quad\forall\xi_1,\xi_2\in\C^p.
\label{sol_stable}
\end{eqnarray}
\end{prop}

\vspace{0.5cm}
The estimate $\alpha\rho(A)\theta<1$ required as one of the 
hypotheses of Proposition \ref{prop_stability} is more frequently
written in an explicit way as follows
\begin{eqnarray*}
\theta_i =2\alpha\rho(A)\int_{I_i}\eta(s)ds
\leq \theta=\sup_{i\in\mathbb{Z}}\theta_i <1.
\end{eqnarray*}
Moreover, we remark two facts. Firstly, 
if we are only interested in the forward (or backward) continuation, 
i.e. only for $t  \geq\tau$ (or $t\le \tau$), then instead of  the condition
$\alpha\rho(A)\theta<1$ we need only
an integration on $I^+_i$ (or $I^-_i$) i.e. the inequality
\begin{eqnarray*}
2\alpha \rho^{+}_{i}(A)\int_{I^+_i}\eta(s)ds\leq \theta^{+}<1
\qquad
\Big( \mbox{ or } 
2\alpha \rho^{-}_{i}(A)\int_{I^+_i}\eta(s)ds\leq \theta^{+}<1\Big),
\end{eqnarray*}
is required. Second, we remark that the
Proposition \ref{prop_stability} generalizes the 
corresponding results obtained by Akhmet \cite{3,4}.

\subsection{Exponential stability}
$ $

The definitions of Lyapunov stability of the solutions of DEPCAG can be given in
the same way as for ordinary differential equations. Let us formulate only one of them.

\begin{defn}
\label{def_exponential_stable}
The zero solution of DEPCAG \eqref{depca_lineal+f} is 
$\sigma$-exponentially stable if for an arbitrary positive 
$\epsilon$, there exists a positive number $\delta=\delta(\tau,\epsilon)$ 
such that $\vert \xi\vert\leq \delta$ implies that
$\vert x(t,\tau,\xi) \vert\leq c\vert \xi\vert e^{-\sigma(t-\tau)}$ 
for all $t\geq\tau\geq 0.$
\end{defn}


Let $Z(t,s)$ be the fundamental matrix  of the 
linear DEPCAG \eqref{depca_lineal} (see \eqref{Z_E}
and \eqref{eq:Z_mas_menos_notation}). By the representations
\eqref{solution} and \eqref{Z_E}, the stability of linear 
system \eqref{depca_lineal} can be analogously expressed 
as theorems for ordinary differential equations \cite{7,13,15,19}. 
An example of this fact is \eqref{1_condition_stability} and
the following theorem (see \cite{4}).

\begin{thm}
\label{thm:expstaboldepcalin}
The zero solution of the linear DEPCAG \eqref{depca_lineal} is 
$\sigma$-exponentially stable if and
only if there exist two positive numbers $c$ 
and $\sigma$ such that
\begin{eqnarray}
\vert Z(t,s)\vert\leq ce^{-\sigma(t-s)},\ \
\mbox{ for }t\geq s\geq 0. 
\label{exp_stable_sol_zero}
\end{eqnarray}
\end{thm}

\vspace{0.5cm}
By Lemma \ref{lema_cotas} 
\begin{eqnarray}
\vert G(t,s)\vert\leq c \rho(A)e^{\sigma\overline{t}}
e^{-\sigma(t-s)}, \mbox{ for } t\geq s\geq 0,
\end{eqnarray}
where $\overline{t}$ is the notation introduced on (S3) (see~\eqref{eq:deftline}).

From the respective stability of the difference system \eqref{recurrence}:
\begin{eqnarray}
z(t_{i+1})=Z(t_{i+1},t_i)z(t_i),\ i\in\mathbb{Z},\label{discrete_equation}
\end{eqnarray}
 whose solutions are given by
\begin{eqnarray}
 z(t_{j+1})=\prod_{k=i}^{j}Z(t_{k+1},t_k)z(t_i)=Z(t_{j+1},t_i)z(t_i),\ \ i< j,
\end{eqnarray}
we can formulate several theorems which provide 
sufficient conditions for the stability of
linear systems of DEPCAG \eqref{depca_lineal}. 
The stability of the solution $z=0$ of the difference
system \eqref{discrete_equation} is deduced from 
boundedness or convergence of $Z(t_{j+1},t_i)$ as $j\rightarrow\infty$. 
Taking on account of formula \eqref{Z_E}, from the hypotheses 
(S1) and (S2), we have, e.g. the following theorem.

\begin{thm}
\label{theo_exp_asym_stable}
Assume that conditions (H1), (H4), (S1) and (S3) 
are fulfilled and the zero
solution of DEPCAG \eqref{depca_lineal} is exponentially 
asymptotically stable if $0< \rho <1$, we have
\begin{eqnarray}
\vert E(t_{k+1},\zeta_k)E(t_{k},\zeta_k)^{-1} \vert
=\vert Z(t_{k+1},t_k)\vert\leq \rho,\ \ k\in \mathbb{N}.
\label{condition_exp_asy_stable}
\end{eqnarray}
\end{thm}

\vspace{0.5cm}
We observe that the
condition \eqref{condition_exp_asy_stable} is the natural ones.
Moreover, we note  that there exist other conditions which 
permits to get similar results. For instance,
under other several additional assumptions, Akhmet \cite{4} consider 
$\vert E(t_{k+1},\zeta_{k+1})E(t_{k+1},\zeta_k)^{-1} \vert\leq \rho$ 
instead of \eqref{condition_exp_asy_stable}. This condition
is equivalent to \eqref{condition_exp_asy_stable} if $A$
and $B$ are constants and scalars and (S3) holds. 
At present, the DEPCAG more studied
satisfy (S3) , see \cite{1,3,6,11,12,30} and also our 
example given below on subsection~\ref{sub:example}.

On the other hand, we note that
some interesting results similar 
to theorems~\ref{thm:expstaboldepcalin}
and \ref{theo_exp_asym_stable}
can be found in \cite{4,11,29,30}.


\subsection{An example for \eqref{depca_lineal} with $A$ and $B$ constants} $ $
\label{sub:example}
$ $

Study the dichotomic character in the following linear DEPCAG:
\begin{eqnarray}
x'(t)=Ax(t)+Bx(\gamma(t))\label{depca_example_2}
\end{eqnarray}
where $A$ and $B$ are fixed real constant matrices such 
that $A^{-1}$ exists and the function $\gamma(t)$ is
defined by sequences $t_i$ and $\zeta_i$ satisfying:
$$\zeta_i-t_i=\nu^{+},\ t_{i+1}-\zeta_i=\nu^{-},\ \ \ i\in\mathbb{Z},$$
where $\nu^{+},\nu^{-}>0$ are fixed numbers. 
Calling,
\begin{eqnarray*}
\Lambda(s)=e^{sA}+A^{-1}(e^{sA}-I_{p})B=e^{sA}[I_{p}+A^{-1}(I_p-e^{-sA})B],
\end{eqnarray*}
we obtain $E(t,\tau)=\Lambda(t-\tau)$. 

Then, to apply Theorem \ref{theo_exp_asym_stable} we must study:
\begin{eqnarray*}
Z(t_{i+1},t_i)=E(t_{i+1},\zeta_i)E(t_i,\zeta_i)^{-1}
=\Lambda(\nu^{-})\Lambda(-\nu^{+})^{-1}=\Lambda_1.
\end{eqnarray*}
For $A=0$, we have $\Lambda(s)=I_p+sB$  and
\begin{eqnarray*}
\Lambda_1 =(I_p+\nu^{-}B)(I_p-\nu^{+}B)^{-1}.
\end{eqnarray*}
For $\gamma(t)=c[t+d/c], c>0,\ c>d$, 
we have $\zeta_i=c,\ t_i=ci-d, \nu^{+}=d,\ \nu^{-}=c-d,
\Lambda_1=\Lambda(c-d)\Lambda(-d)^{-1} $.
In particular, for the famous case of Cooke and 
Wiener \cite{11}, $\gamma(t)=2[t+1/2]$, 
we have $\Lambda_1=\Lambda(1)\Lambda(-1)^{-1}$.
The zero solution is stable if $\rho(\Lambda_1)\leq 1$ 
and exponentially stable if $\rho(\Lambda_1)<1$.
In these conditions,  there exist  $ c \geq 1$ 
and $\kappa$ constants such that:
\begin{eqnarray}
\vert Z(t,\tau)\vert \leq ce^{\kappa\ln(t-\tau)},\ \ t\geq \tau.
\end{eqnarray}
The situation for $t<\tau$ can be treated similarly. 

In the case $A=a\neq 0$ and $B=b$ be scalars, i.e. the equation
\begin{eqnarray}
x'(t)=ax(t)+bx(\gamma(t)),\label{depca_example_2_const}
\end{eqnarray}
we can find explicit conditions on the
coefficients and the sequences for providing exponential 
stability for of zero solution of
\eqref{depca_example_2_const}, see \cite{5,11,29,30}. On the
basis of the previous analysis, we must have
$\rho=\vert \Lambda_1 \vert
=\vert \Lambda(\nu^{-})\Lambda(-\nu^{+})^{-1} \vert<1$, 
either of the inequalities
\begin{eqnarray}
&& -b>a>0,\ \ 
[e^{a\nu^{-}}+e^{-a\nu^{+}}]\left[1+\frac{b}{a}
\right]>2\frac{b}{a}
\label{a}
\\
&&-b>a,\ a<0,\ \ 
[e^{a\nu^{-}}+e^{-a\nu^{+}}]\left[1+\frac{b}{a}
\right]>2\frac{b}{a}
\label{b}
\end{eqnarray}
is sufficient for the zero solution to be exponentially 
stable. For the completely
delayed case $\nu^{+}=0$ and, hence, $t_{i+1}-t_i=\nu^{-}$, 
the conditions \eqref{a} and \eqref{b} are,
respectively, transformed to
\begin{eqnarray*}
-b>a>0,\;\, e^{a\nu^{-}}<\frac{b-a}{b+a}\;\,\mbox{ and }\;\,
-b>a,\;\,\ a<0,\;\, \ \ e^{a\nu^{-}}<\frac{b-a}{b+a}\cdot
\end{eqnarray*}
The stability for the case $a=0$ 
can be also studied.

\section{Stability of solutions for \eqref{depca_lineal+f}}
\label{sec:dos}

\begin{thm}
\label{thm:Stabildepca_linel}
Assume that the hypotheses (H1)-(H3), (S1) and (S3) are 
fulfilled and  the
zero solution of linear DEPCAG \eqref{depca_lineal}
is $\sigma$-exponentially stable.
Moreover, assume the function
$\eta$ on (H3), $\rho(A)$ defined in \eqref{rho_A}
and $\sigma$
satisfy the following requirements
\begin{eqnarray}
\theta:=\sup_{i\in\mathbb{Z}}\Big\{
2ce^{\sigma \overline{t}}\rho(A)
\int_{t_i}^{\zeta_i}\eta(s)ds\Big\}<1,
\qquad
\beta:=\limsup_{t\rightarrow\infty}\frac{\int_{\tau}^{t}\eta(s)ds}{t}<\infty.
\label{eq:hypotesis_c2}
\end{eqnarray}
Then, there exists $\sigma_0=\sigma-\beta\mu c\rho(A)e^{2\sigma\overline{t}}$ 
with $\mu=(2-\theta)(1-\theta)^{-1}$ such that
the zero solution of
\eqref{depca_lineal+f} is $\sigma_0$-exponentially stable.
In particular, if $\eta\in L^{1}([0,\infty))$, then the
$\sigma$-exponential
stability follows and the stability of the zero solution 
of \eqref{depca_lineal} implies the stability of the
solution zero of~\eqref{depca_lineal+f}.
\end{thm}

\begin{proof}
If we consider that $w(t)=w(t,\tau,\xi)$ is a solution of \eqref{depca_lineal+f} 
such that $w(\tau)=\xi$
and, without loss of generality, we assume that 
$t_i\leq \tau<\zeta_i\leq t_{i+1}<\cdots<t_j\leq\zeta_j<t$. 
Then by Corollary~\ref{cor_integral_equation} and
Theorem~\ref{theo_solution}
(see \eqref{integral_equation} and \eqref{var_parametro_Z}), 
we have that
\begin{eqnarray*}
w(t)&=&Z(t,\tau)w(\tau)+\int_{\tau}^{t}G(t,s)f(s,w(s),w(\gamma(s))ds\\
&=&Z(t,\tau)\xi
+Z(t,\tau)\int_{\tau}^{\zeta_{i(\tau)}}\Phi(\tau,s)f(s,w(s),w(\gamma(s))ds
\\
&&
+\sum_{k=i(\tau)+1}^{j(t)}Z(t,t_{k})
\int_{t_{k}}^{\zeta_{k}}\Phi(t_{k},s)f(s,w(s),w(\gamma(s))ds\\
&&+\sum_{k=i(\tau)}^{j(t)-1}Z(t,t_{k+1})
\int_{\zeta_{k}}^{t_{k+1}}\Phi(t_{k+1},s)f(s,w(s),w(\gamma(s))ds
\\
&&
+\int_{\zeta_{j(t)}}^{t}\Phi(t,s)f(s,w(s),w(\gamma(s))ds.
\end{eqnarray*}
By the hypothesis of the $\sigma$-exponential stability
of the solution for \eqref{depca_lineal}, the assumption
(H3), the application of Theorem~\ref{thm:expstaboldepcalin}
and Lemma~\ref{lema_cotas}, 
and by (S3), we deduce the following bound
for $w$
\begin{eqnarray*}
\left|w(t)\right|&\leq&
ce^{-\sigma(t-\tau)}\left|\xi\right|
+c\rho(A)e^{-\sigma(t-\tau)}
\int_{\tau}^{\zeta_{i(\tau)}}
	\eta(s)\Big(|w(s)|+|w(\gamma(s))|\Big)ds\\
&&+\sum_{k=i(\tau)+1}^{j(t)}c\rho(A)e^{-\sigma(t-t_{k})}
\int_{t_{k}}^{\zeta_{k}}\eta(s)\Big(|w(s)|+|w(\gamma(s))|\Big)ds\\
&&+\sum_{k=i(\tau)}^{j(t)-1}c\rho(A)e^{-\sigma(t-t_{k+1})}
\int_{\zeta_{k}}^{t_{k+1}}\eta(s)\Big(|w(s)|+|w(\gamma(s))|\Big)ds\\
&&+\rho(A)\int_{\zeta_{j(t)}}^{t}e^{-\sigma(t-s)}
\eta(s)\Big(|w(s)|+|w(\gamma(s))|\Big)ds.
\\
&\leq&ce^{-\sigma(t-\tau)}\left|\xi\right|+\int_{\tau}^{t}
c\rho(A)e^{-\sigma(t-s-\overline{t})}\eta(s)\left(\left|w(s)\right|
+e^{-\sigma\gamma(s)}\left|w(\gamma(s))\right|
e^{\sigma\gamma(s)}\right)ds,
\end{eqnarray*}
which can be rewritten as follows 
\begin{eqnarray*}
u(t)&\leq& cu(\tau)
+\int_{\tau}^{t}c\rho(A)e^{2\sigma\overline{t}}\eta(s)\left(u(s)
+u(\gamma(s))\right)ds
\quad\mbox{with}\quad
u(t)= \vert w(t)\vert e^{\sigma t}.
\end{eqnarray*}
Now, by virtue of the DEPCAG Gronwall inequality 
given on Lemma \ref{lema_gronwal}, we obtain that 
\begin{eqnarray*}
\left|w(t)\right|
&\leq&
c\exp\Big({-\sigma(t-\tau)+c\mu\rho(A)e^{2\sigma\overline{t}}
\int_{\tau}^{t}\eta(s)ds}\Big)
\quad\mbox{with}\quad
\mu=\frac{2-\theta}{1-\theta}\cdot
\end{eqnarray*}
Hence,
the last inequality combined with \eqref{eq:hypotesis_c2}
proves that the zero solution is $\sigma_0$-exponentially
stable. The other assertions follow similarly. The theorem is proved. 
\end{proof}

If in (H3), the function $\eta$ is constant we have
an interesting result similar to the  ones 
obtained previously by Akhmet in \cite{4} under other conditions.

\begin{cor}
\label{cor_exponentially_stable}
Assume that conditions (H1)-(H3), (S1) and (S3) are fulfilled.
Moreover, in (H3) consider that $\eta$ is constant, i.e.
$\eta(t)\equiv\eta_0$. Suppose that the zero solution of 
the linear DEPCAG \eqref{depca_lineal} is
$\sigma$-exponentially stable and consider that
\begin{eqnarray*}
\theta = 2 c\overline{t}\eta_0 \rho(A) e^{2\sigma\overline{t}}<1,
\quad
\sigma -c\rho(A)\mu \eta_0 e^{2\sigma\overline{t}}=\sigma_0 >0,
\quad
\mu =\frac{2-\theta}{1-\theta}\cdot
\end{eqnarray*}
Then, the solution zero of system \eqref{depca_lineal+f} 
is $\sigma_0$-exponentially stable.
\end{cor}

\section{Bounded solutions for \eqref{depca_lineal+g}} $ $
\label{sec:tres}

The bounded solutions on all of $\mathbb{R}$ of the 
linear nonhomogeneous equation \eqref{depca_lineal+g} can be studied
by considering the convergence of the series 
\begin{eqnarray}
&&\sum_{k=-\infty}^{-1}
\vert Z(0,t_{k+1})\vert 
\int_{\gamma(t_{k})}^{\gamma(t_{k+1})}\vert
\Phi(t_{k+1},s) \vert ds<\infty, 
\label{serie-infty}
\\
&&\sum^{\infty}_{k=0}| Z(0,t_{k+1})|
\int_{\gamma(t_{k})}^{\gamma(t_{k+1})} | \Phi(t_{k+1},s) | ds
<\infty .
\label{serie+infty}
\end{eqnarray}
As $\vert Z(0,t_{k+1})\vert<ce^{-\sigma\vert t_{k+1}\vert}$
estimations of the integrals in \eqref{serie-infty} and 
\eqref{serie+infty} allow give conditions for convergence 
of the above series. For  example, for $t_k=k$ and in general
with (H1) and (S2), the conditions \eqref{serie-infty} and \eqref{serie+infty} hold,
see Lopez-Fenner and Pinto \cite{fenner_pinto}.

\begin{prop}
\label{prop_45}
Let $g:\R\to\C^p$ be a bounded function. The following
assertions with respect to the
solution of the equation \eqref{depca_lineal+g} are valid 
\begin{enumerate}[(a)]
 \item Assume that \eqref{serie-infty}
holds and the solution of \eqref{depca_lineal} 
is $\sigma$-exponentially stable on $\mathbb{R}$, i.e. 
\begin{eqnarray*}
\mbox{$\vert Z(t,s)\vert \leq ce^{-\sigma (t-s)}$ for  $t\geq s.$}
\end{eqnarray*}
Then equation \eqref{depca_lineal+g} has a unique bounded 
solution $y:\R\to\C^p$ defined by
\begin{eqnarray}
y(t)&=&\int_{-\infty}^{t}G(t,s)g(s)ds
=\sum_{k=-\infty}^{i(t)-1}
\int_{t_k}^{t_{k+1}}G_k (t,s) g(s)ds
+\int_{t_{i(t)}}^{t} \Phi(t,s)g(s)ds
\nonumber\\
&=&
\sum_{k=-\infty}^{i(t)-1}Z(t,t_{k})\int_{t_k}^{\zeta_k}\Phi(t_k,s) g(s)ds
+\sum_{k=-\infty}^{i(t)-1}Z(t,t_{k+1})\int_{\zeta_k}^{t_{k+1}}\Phi(t_{k+1},s)  g(s)ds
\nonumber\\
&&+\int_{i(t)}^{t}\Phi(t,s)  g(s)ds.
\label{5.11}
\end{eqnarray}

 \item Assume that \eqref{serie+infty} holds and the following condition  
\begin{eqnarray}
&&\vert Z(t,s)\vert\leq ce^{-\sigma(s-t)},\quad s\geq t\label{Z_exp}
\end{eqnarray}
is  satisfied.
Then the unique bounded solution of equation \eqref{depca_lineal+g} 
on $\mathbb{R}$ is given by
\begin{eqnarray}
y(t)&=&-\int_{t}^{\infty} G(t,s)g(s)ds
= -\int_{t}^{t_{i(t)+1}}\Phi(t,s)g(s)ds
-\sum_{k=i(t)+1}^{\infty}\int_{t_{k}}^{t_{k+1}}G_{k}(t,s)g(s)ds
\nonumber\\
&=& -\int_{t}^{t_{i(t)+1}}\Phi(t,s)g(s)ds
-\sum_{k=i(t)+1}^{\infty}Z(t,t_{k})
\int_{t_{k}}^{\zeta_{k}}\Phi(t_{k},s)g(s)ds 
\nonumber\\
& &-\sum_{k=i(t)+1}^{\infty}Z(t,t_{k+1})
\int_{\zeta_{k}}^{t_{k+1}}\Phi(t_{k+1},s)g(s)ds. 
\label{x_sol_bound_ser+infty}
\end{eqnarray}
\item The map $g \to y_g$ is continuous and satisfies the estimate
\begin{eqnarray}
\vert y_{g}(t)\vert \leq \hat{c}\vert g\vert_{\infty}
\quad 
\mbox{with}
\quad
\vert g\vert_{\infty}=\sup_{t\in\R}|g(t)|
\quad 
\mbox{with}
\quad
\hat{c}=c\rho(A)e^{\sigma\overline{t}},
\label{cota_xf}
\end{eqnarray}
for a positive constant $\hat{c}$ independent on $g$.
\end{enumerate}
\end{prop}

\begin{proof}
(a)
In a standard way, it is not difficult to show that 
$y$ given by \eqref{5.11} is a well defined bounded function  
and is a solution of \eqref{depca_lineal+g}. 
Moreover, for $g$ fixed, the nonhomogeneous 
linear system \eqref{depca_lineal+g}
has a unique bounded solution on all of  $\mathbb{R}$, 
because for $\omega\neq 0$ any solution $Z(t,\tau)\omega$ 
of the homogeneous linear system \eqref{A} is unbounded
as $t\rightarrow -\infty$. Now, we deduce that the unique bounded solution of 
\eqref{depca_lineal+g} is necessarily given by \eqref{5.11}.
Indeed, from \eqref{theo_var_parametro_depca} we have that
any solution $y$ of \eqref{theo_var_parametro_depca}
is given by
\begin{eqnarray}
y(t)=Z(t,0)\omega +\int_{0}^{t} G(t,s)g(s)ds,
\quad\mbox{with}\quad \omega = y(0).
\label{sol_bounded_f}
\end{eqnarray}
Note that 
$\int_{0}^{t}=\int_{0}^{\zeta_{i(0)}}
+\sum_{k=i(0)}^{i(t)-1}\int_{\zeta_k}^{\zeta_{k+1}}
+\int_{\zeta_{i(t)}}^{t}$.
Supposing, by simplicity, that $\zeta_{i(0)}=0$, and that the 
convergence of the series
\begin{eqnarray}
\sum_{k=0}^{-\infty}Z(0,t_{k+1})
\int_{\zeta_{k}}^{\zeta_{k+1}}\Phi(t_{k+1},s)g(s)ds
=v_{-\infty},
\label{serie_igualdad}
\end{eqnarray}
we get
\begin{eqnarray}
\int_{0}^{t}G(t,s)g(s)ds
&=&\sum_{k=0}^{i(t)-1}
\int_{\zeta_{k}}^{\zeta_{k+1}}G_{k}(t,s)g(s)ds
+\int_{\zeta_{i(t)}}^{t} G_{i(t)}(t,s)g(s)ds
\nonumber\\
&=& \sum_{k=0}^{i(t)-1}
\int_{\zeta_k}^{\zeta_{k+1}} Z(t,t_{k+1})\Phi(t_{k+1},s)g(s)ds
+\int_{\zeta_{i(t)}}^{t}\Phi(t,s)g(s)ds
\nonumber\\
&=& Z(t,0)\sum_{k=0}^{i(t)-1}Z(0,t_{k+1})
\int_{\zeta_k}^{\zeta_{k+1}} \Phi(t_{k+1},s)g(s)ds
+\int_{\zeta_{i(t)}}^{t}\Phi(t,s)g(s)ds
\nonumber\\
&=&Z(t,0)\left( \sum_{k=0}^{i(-\infty)}
-\sum_{k=i(t)}^{i(-\infty)}\right) Z(0,t_{k+1})
\int_{\zeta_k}^{\zeta_{k+1}} 
\Phi(t_{k+1},s)g(s)ds\nonumber\\
&& +\int_{\zeta_{i(t)}}^{t}
\Phi(t,s)g(s)ds. 
\label{sol_serie_}
\end{eqnarray}
Then, introducing \eqref{serie_igualdad} and \eqref{sol_serie_} in
\eqref{sol_bounded_f}, we have
\begin{eqnarray*}
y(t)&=&Z(t,0)[\omega+v_{-\infty}]
+\sum_{k=-\infty}^{i(t)-1}Z(t,t_{k+1})\int_{\zeta_k}^{\zeta_{k+1}} 
\Phi(t_{k+1},s)g(s)ds+\int_{\zeta_{i(t)}}^{t}\Phi(t,s)g(s)ds\\
&=&\int_{-\infty}^{t}G(t,s)g(s)ds,
\end{eqnarray*}
by \eqref{var_parametro_Z_} and taking $\omega=-v_{-\infty}$
we avoid the homogenous unbounded solution  in the first term 
and \eqref{5.11} is deduced. 

(b) The proof of \eqref{x_sol_bound_ser+infty} when \eqref{Z_exp} holds, follows similarly.

(c) To prove that $y$ satisfies \eqref{cota_xf}, we apply the properties
of $G_k$ and $Z$ in \eqref{5.11} and \eqref{x_sol_bound_ser+infty}.

\end{proof}

\vspace{0.5cm}
We remark that estimations of the type 
$\vert Z(0,t_{k+1})\vert<ce^{-\sigma\vert t_{k+1}\vert}$ 
imply \eqref{serie-infty} and \eqref{serie+infty}. Moreover,
we note that these kind of estimates are  valid for example
for the particular case $t_k=k$ and in general 
when  (H2) and (S2) hold, see \cite{fenner_pinto} for details.
Then we have the following corollary.

\begin{cor}
Let $g:\R\to\C^p$ be a bounded function.
The results of Proposition~\ref{prop_45} are valid 
if the hypotheses (H2) and (S3) are fulfilled.
\end{cor}

\section{Asymptotic Equivalence, Ordinary and Exponential Dichotomies.}
\label{sec:cuatro}

In Proposition \ref{prop_stability} we have studied a uniform 
stability given by \eqref{sol_stable}. Its dichotomic extension 
carries us to the ordinary dichotomy which includes an unstability.

\subsection{Ordinary dichotomy and Green matrix}
$ $

\begin{defn} 
\label{def:ordinary_dich}
The linear DEPCAG \eqref{depca_lineal} has an
ordinary dichotomy if there exists a projection $P$ and 
a positive  $c$ such that $|Z_P(t,s) |\leq c$ for all $(t,s)\in\R^2$,
where the Green function
$Z_{P}:\R^2\to\C^{p\times p}$ is defined by
\begin{eqnarray}
Z_{P}(t,s)=\begin{cases}
Z(t,0)PZ(0,s), &t\geq s,\\
-Z(t,0)(I-P)Z(0,s), & t<s,
\end{cases}
\label{eq:zp_function}
\end{eqnarray}
for a given a projection matrix $P\in \C^{p\times p}$.
\end{defn}

\begin{defn}
\label{def:green_matrix_Gtilde}
Consider that $Z_{P}$ denotes the function defined on \eqref{eq:zp_function}.
Then, the Green matrix type $\tilde{G}:\R^2\to\C^{p\times p}$ is defined
as follows
\begin{eqnarray*}
\tilde{G}(t,s)
&:=&\left\langle Z_{P}(t,\cdot),\Phi(\cdot,s)\right\rangle \\
&:=&Z_{P}(t,\tau)\cdot\tilde{\Phi}_{-}(\tau,s)
\\
&& +\sum_{k=i(\tau)+1}^{i(t)}Z_{P}(t,t_k)\tilde{\Phi}_{+}(t_{k},s)
+\sum_{k=i(\tau)}^{i(t)-1}Z_{P}(t,t_{k+1})\tilde{\Phi}_{-}(t_{k+1},s)+\tilde{\Phi}(t,s),
\end{eqnarray*}
where $\tilde{\Phi}_{\pm}(t_{k},s)=\Phi(t_{k},s)1_{I^\pm_k}(s)$ and 
$\tilde{\Phi}(t,s)=\Phi(t,s)1_{[\xi_{i(t)},t]}$. Here $1_\mathcal{A}$ denotes
the standard characteristic function of the set 
$\mathcal{A}\subset\mathcal{U}$, i.e. $1_\mathcal{A}(s)=1$ for $s\in \mathcal{A}$
and $1_\mathcal{A}(s)=0$ for $s\notin \mathcal{A}$.
\end{defn}

By \eqref{rho_A} we can deduce that 
\begin{eqnarray}
\mbox{$\vert \tilde{G}(t,s) \vert \leq \tilde{c}$ 
for all $(t,s)\in\R^2$ 
with $\tilde{c}=c\rho(A)$.}
\label{eq:defctilde}
\end{eqnarray}
Now, for an integrable function 
$g:[\tau,\infty)\rightarrow \mathbb{C}^{n}$, we have 
that the solution of \eqref{depca_lineal+g} is given by
\begin{eqnarray*}
y(t)
=\int_{\tau}^{\infty}\tilde{G}(t,s)g(s)ds
=\int_{\tau}^{t}\tilde{G}(t,s)g(s)ds 
+\int_{t}^{\infty}\tilde{G}(t,s)g(s)ds
=:y_{+}(t)+y_{-}(t)
\end{eqnarray*}
where
\begin{eqnarray*}
y_{+}(t)
&=& \int_{\tau}^{t} <Z_P(t,\cdot);\Phi(\cdot ,s)>g(s)ds \\
&=&Z_{P}(t,\tau)
\int_{\tau}^{\zeta_{i(\tau)}}\Phi(\tau ,s)g(s)ds
+\sum_{k=i(\tau)+1}^{i(t)}Z_{P}(t,t_{k})
\int_{t_{k}}^{\zeta_{k}}\Phi(t_{k},s)g(s)ds
\\
&&+\sum_{k=i(\tau)}^{i(t)-1}Z_{P}(t,t_{k+1})
\int_{\zeta_{k}}^{t_{k+1}}\Phi(t_{k+1},s)g(s)ds
+\int_{\zeta_{i(t)}}^{t}\Phi(t,s)g(s)ds
\\
y_{-}(t)
&=& -\int^{\infty}_{t} <Z_P(t,\cdot);\Phi(\cdot ,s)>g(s)ds \\
&=&-Z_{P}(t,t_{i(t)+1})\int_{t}^{t_{i(t)+1}}\Phi(t_{i(t)+1},s)g(s)ds
-\sum_{k=i(t)+1}^{\infty}Z_{P}(t,t_{k})
\int_{t_{k}}^{\zeta_{k}}\Phi(t_{k},s)g(s)ds
\\
&&-\sum_{k=i(t)+1}^{\infty}Z_{P}(t,t_{k+1})
\int_{\zeta_{k}}^{t_{k+1}}\Phi(t_{k+1},s)g(s)ds.
\end{eqnarray*}
In particular, if  we have an ordinary stability with $P=I$, 
$\vert Z(t,s)\vert \leq c,$ $c\geq 1$, for $t\geq s,$
the special bounded solution on $\mathbb{R}$ of \eqref{depca_lineal+g} is
given by 
$y^{+}_g(t) = \int _{-\infty}^{t}\tilde{G}(t,s)g(s)ds$
and in the unstable situation $P=0$,
$\vert Z(t,s)\vert \leq c,$ $c\geq 1,$ for $t\leq s,$
the special bounded solution on $\mathbb{R}$ of \eqref{depca_lineal+g} is
given by $y^{-}_g(t)=\int_{t}^{\infty} \tilde{G}(t,s)g(s)ds$
and the analogous of the bound \eqref{cota_xf} is stated as follows
\begin{eqnarray*}
 \Vert y^{\pm}_g \Vert_{\infty}\leq \tilde{c} \Vert g\Vert_{1}
\end{eqnarray*}
with $\tilde{c}$ given in \eqref{eq:defctilde}.

Now, we prove the asymptotic equivalence between system 
\eqref{depca_lineal+g} and \eqref{depca_g+f} when the perturbation
$f$ is integrable, i.e. $\eta,f(t,0,0)\in L^1([t_0,\infty)).$

\begin{thm}
\label{thm_L_fintegrable}
Assume that the linear system 
\eqref{depca_lineal}
has an ordinary dichotomy with
projection $P$ on $[t_0,\infty)$ 
and the hypotheses (H1)-(H4) are fulfilled. 
Moreover, assume that instead of (H3)-(ii) the 
condition $f(t,0,0)\in L^{1}([t_0,\infty))$
is satisfied and the function $\eta$ in (H3)-(iii) 
is belonging $L^{1}([t_0,\infty))$.
Then there exists a homeomorphism
between the bounded solutions of the 
linear system \eqref{depca_lineal+g} and
the bounded solutions of the quasilinear system
\eqref{depca_g+f}.
Moreover, $|y(t)-v(t)|\to 0$ as
$t\to \infty$ if $Z(t,0)P\to 0$ as $t\to \infty$.
\end{thm}

\begin{proof}
Consider that $y$ is a bounded solution 
of \eqref{depca_lineal+g} and denote by 
$BC([t_0,\infty),\C^p)$ the space of 
bounded continuous function with the topology 
defined by the norm $\Vert y\Vert=\sup_{s\geq t^{*}_0}\vert y(s)\vert$ 
with $t^{*}_0=\min\{t_0,\gamma(t_0)\}$.
Now, we consider 
the operator
$\mathcal{A}: BC([t_0,\infty),\C^p)\rightarrow BC([t_0,\infty),\C^p)$
defined by
\begin{eqnarray*}
(\mathcal{A}v)(t)=y(t)+\int_{t_0}^{\infty}\tilde{G}(t,s)f(s,v(s),v(\gamma(s)))ds.
\end{eqnarray*}
From \eqref{eq:defctilde},
(H3)-(i) and (H3)-(iii),
we can prove  that $\mathcal{A}$
is a contraction for $t_0$ sufficiently  large, since
\begin{eqnarray*}
\Vert \mathcal{A}v-y\Vert &\leq& \tilde{c}
\int_{t_0}^{\infty}\Big(\eta(s)\Big[\vert v(s)\vert 
+\vert v(\gamma(s))\vert\Big]+\vert f(s,0,0)\vert\Big)ds\\
\Vert \mathcal{A}v_1-\mathcal{A}v_2\Vert 
&\leq& \tilde{c}\int_{t_0}^{\infty}\eta(s)
\Big[\vert v_1(s)-v_2(s)\vert 
+\vert v_1(\gamma(s))-v_2(\gamma(s))\vert\Big]ds\\
&\leq& \beta \Vert v_1-v_2 \Vert,
\end{eqnarray*}
with 
$\beta= 2\tilde{c}\int_{t_0}^{\infty}\eta(s)ds<1$.
Hence, the integral equation
\begin{eqnarray}
v(t)=y(t)+\int_{t_0}^{\infty} \tilde{G}(t,s)
f\big(s,v(s),v(\gamma(s))\big)ds
\label{integral_equation_f+q}
\end{eqnarray}
has a unique bounded solution and this solution is
the unique bounded continuous solution of \eqref{depca_g+f}.
Then, summarizing, we have that
for any bounded continuous $y$ of \eqref{depca_lineal+g}, 
the integral equation \eqref{integral_equation_f+q}
has a unique bounded continuous solution $v$ which is the 
solution of \eqref{depca_g+f}. 
Reciprocally, by  the properties of $f$, $\tilde{G}$ and
the integral equation is
straightforward to deduce that
if $v$ is a bounded solution 
of \eqref{depca_g+f} then $y$ defined by 
\eqref{integral_equation_f+q} is a bounded
solution of \eqref{depca_lineal+g}.
Moreover, the correspondence $y\rightarrow v$ 
is bicontinuous, since the estimates 
\begin{eqnarray*}
\Vert v_1 -v_2 \Vert  
\leq \Vert y_1 -y_2 \Vert 
+\beta \Vert v_1 -v_2\Vert ,\qquad  
\Vert y_1 -y_2 \Vert \leq 
\Vert v_1 -v_2\Vert 
+\beta \Vert v_1 -v_2 \Vert ,
\end{eqnarray*}
gives
\begin{eqnarray*}
(1+\beta)^{-1}\Vert y_1 -y_2 \Vert\leq 
\Vert v_1 -v_2\Vert\leq 
(1-\beta)^{-1}\Vert y_1 -y_2 \Vert.
\end{eqnarray*}
Finally,
by \eqref{eq:defctilde} and the properties of $f$, 
we deduce that 
for any $\epsilon>0$  there exists $T\geq t_0$ such that
\begin{eqnarray*}
\left\vert \int_{T}^{\infty} 
\tilde{G}(t,s)f(s,v(s),v(\gamma(s)))ds\right\vert 
\leq \tilde{c}\int_{T}^{\infty}
\left( 2\Vert v\Vert\eta(s)
+\vert f(s,0,0)\vert\right) ds<\epsilon
\end{eqnarray*}
and
\begin{eqnarray*}
\vert \mathcal{A}v(t)-y(t) 
\vert \leq \vert Z(t,0)P \vert \int_{t_0}^{T}
\left\vert \tilde{G}(0,s)f(s,v(s),v(\gamma(s)))
\right\vert ds  +\epsilon.
\end{eqnarray*}
Hence, the condition $\vert Z(t,0)P\vert\rightarrow 0$ as 
$t\rightarrow \infty$ implies that 
$\vert y(t)-v(t)\vert \rightarrow 0$ as
$t\rightarrow\infty$.
\end{proof}

\vspace{0.5cm} 
We note that under the conditions of the Theorem \ref{thm_L_fintegrable} 
the solutions of \eqref{depca_lineal+g} and \eqref{depca_g+f} 
are defined for all $t\geq t_0$ are uniquely  determined by
their initial values and depend  continuously  on these initial 
values over any bounded interval. Then, we  have the continuity property
on the original interval $[t_0,\infty)$ and not only for $t_0$ 
sufficiently large. Moreover, we note that
$f$ is linear, the correspondence  $y\rightarrow v$
is linear and homogeneous.
Meanwhile, related with the hypothesis
'the linear system \eqref{depca_lineal+g} has an ordinary dichotomy',
we have that this condition is  satisfied, for instance, if $\Vert B\Vert$ 
is small enough and $A(t)=A$ is a constant matrix whose 
characteristic  roots with zero real parts of simple type
or if $A(t)$ is periodic, where the characteristic exponent 
with zero real part of simple type.

\vspace{0.5cm} 

On the other hand, we note that the  equation $x'(t)=0$ with $P=0$ 
satisfies the hypothesis of  Theorem \eqref{thm_L_fintegrable},
then we can deduce a result for equation \eqref{B}.

\begin{cor}
\label{eq:homeo_cp}
Consider the hypotheses of Theorem~\ref{thm_L_fintegrable}.
Then, there exists a homeomorphism between $\C^p$ and 
the bounded solutions of \eqref{depca_lineal+f}. Moreover,
any solution $u$ is convergent to some $\hat{u}\in\C^p$ 
as $t\rightarrow \infty$ and for every $\hat{u}\in \C^p$ 
there exists a unique solution $u$ of \eqref{B} 
such that $u(t)\rightarrow \hat{u}$ as $t\rightarrow \infty$.
\end{cor}

\subsection{Exponential dichotomy and Green matrix}
$ $

For the conditional asymptotic stability we have
\begin{defn} 
\label{def:exp_dich}
Consider that $Z_{P}$ denotes the function defined on \eqref{eq:zp_function}.
Then, the linear DEPCAG \eqref{depca_lineal} has a 
$\sigma$-exponential dichotomy if there exists a
projection matrix $P$ and a positive constant $c>0$ such that
$\vert Z_P(t,s)\vert \leq ce^{-\sigma \vert t-s\vert}$.
\end{defn}

Now, the Green matrix in Definition~\ref{def:exp_dich}
satisfies the estimate
\begin{eqnarray}
\vert \tilde{G}_p (t,s)\vert \leq \hat{c}
e^{-\sigma\vert t-s\vert},
\quad t,s\in\mathbb{R},
\quad \hat{c}=c\rho(A)e^{\sigma \overline{t}}.
\label{eq:green_estimate_expdich}
\end{eqnarray}

\begin{rem}
G. Papaschinopoulos \cite{24,Papas2,Papaschinopoulos} propose to 
define an exponential dichotomy for linear DEPCAG \eqref{depca_lineal}
when the difference equation \eqref{recurrence} has an 
exponential dichotomy. Definition~\ref{def:exp_dich} 
is rather a natural notion of exponential dichotomy.
However, if we take $A(t)=0, B(t)=\mbox{\rm diag}(\lambda_{1}(t),\lambda_2 (t)),\ 
\lambda_1(t)=-\frac{2}{\pi}+\sin(2\pi t)$, 
$\lambda_2 (t)=-\lambda_1 (t)$, $t_n=n$ for all 
$n\in\mathbb{Z}$, $\int_{n}^{n+\delta} \lambda_1(\xi)d\xi
=-\frac{1}{2\pi}(4\delta-1+\cos(2\pi\delta)) $ 
for all $\delta\in [0,1]$ then the difference 
equation \eqref{recurrence} has an exponential
dichotomy with projection $P=\mbox{\rm diag}(1,0)$ 
but there is no $P$ such that the estimation for $Z_P$
on Definition~\ref{def:exp_dich} 
is satisfied. Indeed, for $t-[t]<1/2$, 
$\int_{[t]}^{t}\lambda_{1}(s)ds \geq 0$ and is negative
for $t-[t]>1/2$, while $\int_{[t]}^{t}\lambda_2(s)ds$ 
satisfies the same with contrary sign. 
However, $\int_{[t]}^{t}\lambda_1(s)ds=0
=\int_{[t]}^{t}\lambda_2(s)ds$ for $t-[t]=1/2$. 
\end{rem}

Notice that a dichotomy condition on the ordinary 
differential equation \eqref{A} implies an exponential 
dichotomy on the difference equation \eqref{recurrence}
when $\vert B(t)\vert$ is small enough
\cite[Proposition 2]{fenner_pinto}. However, 
 an exponential dichotomy for the difference equation 
 on \eqref{recurrence} is not a necessary condition for
 an exponential dichotomy for the ordinary 
 differential system \eqref{A}. In fact,
 let's consider $t_n =n$ $A(t)=0$ and
 $B(t)=\mbox{diag}(-\frac{3}{2},\frac{1}{2})$. 
 Then the exponential dichotomy for difference 
 system \eqref{recurrence} is satisfied, with no exponential
 dichotomy for the ordinary differential system \eqref{A}.

Assume the convergence of the series
\begin{eqnarray}
&&\sum_{k=-\infty}^{0}  \left|PZ(0,t_{k+1}) 
\int_{\gamma(t_k)}^{\gamma(t_{k+1})} 
\Phi(t_{k+1},s) ds\right|<\infty,
\label{serie--infty} \\
&&\sum_{k=0}^{\infty} \left| (I-P)Z(0,t_{k+1}) 
\int_{\gamma(t_k)}^{\gamma(t_{k+1})} 
\Phi(t_{k+1},s) ds\right|<\infty. 
\label{serie_+infty}
\end{eqnarray} 
Note that $\vert PZ(0,t_{k+1})\vert$, 
$\vert (I-P)Z(0,t_{k+1})\vert\leq ce^{-\sigma\vert t_{k+1}\vert}$ 
and estimations of the integrals in the above series establish 
conditions for its convergence.

For example, $t_k=rk, 0<r<1$  
and in general \eqref{serie--infty} and 
\eqref{serie_+infty} are true if $(H2)$ and $(S2)$ hold. 
See Lopez-Fenner-Pinto \cite{fenner_pinto}

We have the fundamental result about bounded solution on
$\mathbb{R}$ of the linear non homogeneous DEPCAG.

\begin{thm}
\label{theo_sigma_exp_unique_sol}
Let $g:\R\to\C^p$ be a bounded function.
Assume that the linear DEPCAG \eqref{depca_lineal} 
has a $\sigma$-exponential dichotomy such that  
\eqref{serie--infty} and \eqref{serie_+infty}
hold. Then 
there exists $y:\R\to\C^p$ a unique bounded solution  
of the non-homogeneous linear DEPCAG \eqref{depca_lineal+g}
are defined by 
\begin{eqnarray}
y_g(t)=\int_{-\infty}^{\infty}\tilde{G}(t,s)g(s)ds
=\int_{-\infty}^{t}\tilde{G}(t,s)g(s)ds
+\int_{t}^{\infty}\tilde{G}(t,s)g(s)ds.
\nonumber
\end{eqnarray}
Moreover the correspondence 
$g\rightarrow y_g$ defines a Lipschitz 
continuous operator on $\mathcal{B}(\R,\C^p)$
and $\vert y_g\vert_{\infty}\leq \hat{c}\vert g \vert_{\infty}$
with $\hat{c}$ given by \eqref{eq:green_estimate_expdich}.
\end{thm}

\begin{proof}
We proceed as in the proof of  Proposition~\ref{prop_45} 
by noticing that $Z(t,0)$ can be decomposed as follows
$Z(t,0)=Z(t,0)P+Z(t,0)(I-P).$ Moreover, 
in this case, we get that
\begin{eqnarray*}
\omega =\sum_{k=-\infty}^{0}  PZ(0,t_{k+1}) 
\int_{\gamma(t_k)}^{\gamma(t_{k+1})} \Phi(t_{k+1},s)ds
 +\sum_{k=0}^{\infty}  (I-P)Z(0,t_{k+1}) 
 \int_{\gamma(t_k)}^{\gamma(t_{k+1})} \Phi(t_{k+1},s)ds,
\end{eqnarray*}
instead  of $\omega=-v_{-\infty}$.
\end{proof}

We note that estimations of the type 
$\vert PZ(0,t_{k+1})\vert$, 
$\vert (I-P)Z(0,t_{k+1})\vert\leq ce^{-\sigma\vert t_{k+1}\vert}$ 
implies the convergence
of the series defined on \eqref{serie--infty} and \eqref{serie_+infty}.
Furthermore,
we observe that these kind of estimates are  valid for example
for the particular case $t_k=rk$ with $r\in (0,1)$ 
and in general 
when  (H2) and (S2) hold, see \cite{fenner_pinto} for details.
Then we have the following corollary.

\begin{cor}
Let $g:\R\to\C^p$ be a bounded function.
The results of Theorem~\ref{theo_sigma_exp_unique_sol} are valid 
if the hypotheses (H2) and (S2) are fulfilled.
\end{cor}

\vspace{0.5cm}
Now we study bounded perturbations which cannot be studied
with ordinary dichotomy.

\begin{thm}
\label{teo:dsolgendecag}
Assume that the linear system \eqref{depca_lineal} has
a  $\sigma$-exponential dichotomy such that series 
\eqref{serie--infty} and \eqref{serie_+infty} hold
and $f$ satisfies the hypothesis (H3)
with $\eta$ such that $\vert \eta(t)\vert\leq \eta_0$
for all $t\in [t_0,\infty)$.
Moreover, consider that (H2)  and the inequality 
$\beta=2\hat{c}\eta_0(\sigma)^{-1}<1$,
with $\hat{c}$ defined in \eqref{eq:green_estimate_expdich},
are satisfied.
Then, for any $\xi\in P\C^{p}$ 
the nonlinear equation \eqref{depca_lineal+f} has
a unique bounded solution $w$ on $[t_0,\infty)$ 
with $Pw(t_0)=\xi$. Furthermore, the correspondence 
$\xi\rightarrow w$ is continuous and 
any bounded solution $w$ of the equation
\eqref{depca_lineal+f} for $t\geq 0$, satisfies
\begin{eqnarray}
 |w(t)|\le (1-\beta)^{-1} c|\xi| e^{-\sigma_0 t},
 \quad t\ge 0,
 \label{eq:teo:dsolgendecag:0}
\end{eqnarray}
where 
\begin{eqnarray}
 \sigma_0=\sigma-\mu(1-\beta)^{-1}\hat{c}\eta_0e^{\sigma \overline{t}}>0,
 \quad
 \mu =\frac{2-\theta}{1-\theta},
 \quad
 \theta =2c\overline{t}\eta_0\rho(A)e^{2\sigma\overline{t}}<1,
\label{eq:teo:dsolgendecag:1}
\end{eqnarray}
with $\eta_0$ sufficiently small.

\end{thm}

\begin{proof}
The analysis of the bounded solutions for equation \eqref{depca_lineal+f}
is related with the nonlinear operator 
$\mathcal{D}: BC([t_0,\infty),\C^p)\rightarrow BC([t_0,\infty),\C^p)$
defined as follows
\begin{eqnarray*}
(\mathcal{D}w)(t)=Z(t,t_0)\xi+\int_{t_0}^{\infty}\tilde{G}(t,s)f(s,w(s),w(\gamma(s)))ds
\end{eqnarray*}
Now, by the hypothesis  $\xi\in P\C^{p}$ we deduce 
that the operator $\mathcal{D}$ is equivalent to the
operator $\mathcal{A}$ defined in
the proof of  Theorem~\ref{thm_L_fintegrable},
since $(\mathcal{D}w)(t)=(\mathcal{A}w)(t)$
for all $t\in[t_0,\infty)$ by considering that 
$y(t)=Z(t,t_0)P\xi=Z(t,t_0)\xi$
and $v(t)=w(t)$. Then, to prove the properties
of $\mathcal{D}$ we proceed as in
the proof of Theorem~\ref{thm_L_fintegrable}.
Indeed, consider the integral equation 
\eqref{integral_equation_f+q} 
with $y(t)=Z(t,t_0)P\xi=Z(t,t_0)\xi$
and $v(t)=w(t)$.
Again $\mathcal{A}$ (or equivalently $\mathcal{D}$) is a contraction
since 
\begin{eqnarray*}
\Vert \mathcal{A}w_1-\mathcal{A}w_2 \Vert \leq 
\beta \Vert w_1-w_2\Vert
\quad
\mbox{with}
\quad
\beta=\frac{2\hat{c}\eta_0}{\sigma}<1,
\end{eqnarray*}
and then there exists a unique bounded continuous
solution $w$ of \eqref{depca_lineal+f}. 
The correspondence $\xi\rightarrow w_\xi$ is continuous since as
in Theorem~\ref{thm_L_fintegrable}
\begin{eqnarray*}
\Vert y_{\xi_1}-y_{\xi_2} \Vert\leq c\vert \xi_1-\xi_2 \vert 
+ \beta \Vert y_{\xi_1}-y_{\xi_2} \Vert
\quad
\mbox{and}
\quad
\Vert y_{\xi_1}-y_{\xi_2} \Vert \leq c(1-\beta)^{-1}\vert \xi_1-\xi_2 \vert.
\end{eqnarray*}
Now, to prove that any bounded solution of the equation
\eqref{depca_lineal+f} for $t\geq 0$ 
converges exponentially to $0$ as $t\rightarrow \infty$,
we denote by $w$ a bounded solution of \eqref{depca_lineal+f}
and define the function
$$z(t)=w(t)-\mathcal{A}w(t).$$
Note that $z$
is well defined, continuous and bounded on $[0,\infty)$. 
Moreover, $z$ is the solution of the linear DEPCAG \eqref{depca_lineal}
satisfying $Pz(0)=0$ and hence $z(t)=Z(t,0)(I-P)z(0)$ 
which is bounded only if $(I-P)z(0)=0$. Then $z(t)\equiv 0$,
which implies that 
\begin{eqnarray}
w(t)=Z(t,0)\xi+\int_{0}^{\infty} \tilde{G}(t,s)
f(s,w(s),w(\gamma(s)))ds 
\label{integral_equation_f}
\end{eqnarray}
and $w(t)\rightarrow 0$ as $t\rightarrow \infty$. 
Indeed, let $\theta\in (\beta,1)$ and
considering that  $\varlimsup_{t\rightarrow \infty}\vert w(t)\vert=\ell>0$,
then $\vert w(t)\vert\leq \theta^{-1}\ell$ 
for $t\geq T$ and by \eqref{integral_equation_f} we deduce that
that
\begin{eqnarray*}
\vert w(t)\vert\leq 
\vert Z(t,0)\vert\vert \xi\vert
+\vert Z(t,0)P\vert \left\vert 
\int_{0}^{T} Z_{P}(0,s)f(s,w(s),w(\gamma(s)))ds 
\right\vert+\beta\theta^{-1}\ell
\end{eqnarray*}
which letting $t\rightarrow\infty$, 
gives $\ell\leq \beta\theta^{-1}\ell$ 
which is impossible; and hence $\ell=0$.

Finally, from the integral equation \eqref{integral_equation_f} we get
\begin{eqnarray}
|w(t)|&\le& ce^{-\sigma t}|\xi|
+\hat{c}\eta_0
\int_{0}^{t} e^{-\sigma(t-s)}(|w(s)|+|w(\gamma(s))|)ds
\nonumber\\
&&
+\hat{c}\eta_0
\int_{t}^{\infty} e^{-\sigma(s-t)}(|w(s)|+|w(\gamma(s))|)ds.
\label{eq:teorema:eq:2}
\end{eqnarray}

Define 
\begin{eqnarray*}
m(t)=\sup_{s\ge t}|w(s)|.
\end{eqnarray*}
Since $w(t)\to 0$ as $t\to\infty,$ $m(t)$ exists and is
monotone nonincreasing. Moreover, for each $t$ there exists $\tilde{t}\ge t$
such that 
\begin{eqnarray}
\mbox{$m(t)=m(\tilde{t})$ and $m(s)=m(t)=m(\tilde{t})$
for $s\in [t,\tilde{t}].$}
\label{eq:teorema:eq:3}
\end{eqnarray}
Thus \eqref{eq:teorema:eq:2} with $t=\tilde{t}$ yields 
\begin{eqnarray*}
m(\tilde{t})&\le& ce^{-\sigma\tilde{t}}|\xi|
+\hat{c}\eta_0
\int_{0}^{\tilde{t}} e^{-\sigma(\tilde{t}-s)}(|m(s)|+|m(\gamma(s))|)ds
\nonumber\\
&&
+m(\tilde{t})\hat{c}\eta_0
\int_{\tilde{t}}^{\infty} e^{-\sigma(s-\tilde{t})}ds
\end{eqnarray*}
or by \eqref{eq:teorema:eq:3} 
\begin{eqnarray*}
m(t)&\le& ce^{-\sigma t}|\xi|
+\hat{c}\eta_0
\int_{0}^{t} e^{-\sigma(t-s)}
\Big(|m(s)|+e^{-\sigma\gamma(s)}e^{\sigma\gamma(s)}|m(\gamma(s))|\Big)ds
+\beta m(t),
\end{eqnarray*}
since $\beta=2\hat{c}\eta_0(\sigma)^{-1}<1,$ $M(t)=e^{\sigma t}m(t)$
satisfies
\begin{eqnarray*}
M(t)&\le& (1-\beta)^{-1}c|\xi|
+
(1-\beta)^{-1}\hat{c}\eta_0 e^{\sigma \overline{t}}
\int_{0}^{t}
\Big(|M(s)|+|M(\gamma(s))|\Big)ds,
\end{eqnarray*}
which by DEPCAG Gronwall inequality Lemma~\ref{lema_gronwal} gives
\begin{eqnarray*}
M(t)&\le& (1-\beta)^{-1}c|\xi|
\exp\Big(t\mu
(1-\beta)^{-1}\hat{c}\eta_0e^{\sigma\overline{t}}
\Big)
\end{eqnarray*}
or
\begin{eqnarray*}
m(t)&\le& (1-\beta)^{-1}c|\xi|
\exp\Big(-\Big[\sigma-\mu
(1-\beta)^{-1}\hat{c}\eta_0e^{\sigma\overline{t}}
\Big]t\Big),
\end{eqnarray*}
where $\mu$ and $\theta$ are given by \eqref{eq:teo:dsolgendecag:1}.
Thus  \eqref{eq:teo:dsolgendecag:0} is proved.
\end{proof}

%
%

\section*{Acknowledgement}

A. Coronel 
thanks for the  support of research projects GI 153209/C
and DIUBB GI 152920/EF
at Universidad del B{\'\i}o-B{\'\i}o, Chile.
M. Pinto thanks for the  support of FONDECYT 1120709.

\end{document}